\documentclass[a4paper]{amsart}

\usepackage{epsfig}
\usepackage{amsthm,amsfonts}
\usepackage{amssymb,graphicx,color}
\usepackage[all]{xy}
\usepackage{cancel} %\cancel{}, \bcancel{}, \xcancel{}
\usepackage{verbatim}
\usepackage{hyperref}
\usepackage{enumitem}

\newtheorem{theorem}{Theorem}[section]
\newtheorem*{theorem*}{Theorem}

\newtheorem{corollary}[theorem]{Corollary}

\newtheorem{definition}[theorem]{Definition}
\newtheorem{example}[theorem]{Example}
\newtheorem{claim}{Claim}[theorem]
\newtheorem{remark}[theorem]{Remark}

\newcommand{\R}{\mathbb{R}}
\newcommand{\C}{\mathbb{C}}

\newcommand{\N}{\mathbb{N}}

\begin{document}

\title[Differential invariance of the multiplicity]
{Differential invariance of the  multiplicity of real and complex analytic sets}
% \title[Blow-spherical equivalence and multiplicity]
% {Blow-spherical equivalence and multiplicity of real analytic sets}

\author[J. Edson Sampaio]{Jos\'e Edson Sampaio}

\address{Jos\'e Edson Sampaio:  Departamento de Matem\'atica, Universidade Federal do Cear\'a,
	      Rua Campus do Pici, s/n, Bloco 914, Pici, 60440-900, 
	      Fortaleza-CE, Brazil.  
%           E-mail: {\tt edsonsampaio@mat.ufc.br}
%           \newline
% 		  (2) BCAM - Basque Center for Applied Mathematics,
% 	      Mazarredo, 14 E48009 Bilbao, Basque Country - Spain.   
% 	      E-mail: {\tt esampaio@bcamath.org}    
}
\email{edsonsampaio@mat.ufc.br}

\thanks{The author was partially supported by CNPq-Brazil grant 303811/2018-8.
%, by the ERCEA 615655 NMST Consolidator Grant and also by the Basque Government through the BERC 2018-2021 program and Gobierno Vasco Grant IT1094-16, by the Spanish Ministry of Science, Innovation and Universities: BCAM Severo Ochoa accreditation SEV-2017-0718. 
% The author was partially supported in the initial part of this paper by FUNCAP.
}

%\thanks{The authors were partially supported by CNPq-Brazil}

\begin{abstract}
This paper is devoted to proving the differential invariance of the multiplicity of real and complex analytic sets. In particular, we prove the real version of Gau-Lipman's Theorem, i.e., it is proved that the multiplicity mod 2 of real analytic sets is a differential invariant. We prove also a generalization of Gau-Lipman's Theorem.
\end{abstract}

\maketitle
\section{Introduction}

% We motivated by recent results on multiplicity, regularity as well as the Blow-Spherical Geometry of complex analytic sets in \cite{Sampaio:2017}, we do such studies in the case of real analytic sets.
In 1983, Y.-N. Gau and J. Lipman in the paper \cite{Gau-Lipman:1983} proved the following result about the differential invariance of the multiplicity of complex analytic sets (see \cite{Chirka:1989} for a definition of multiplicity of complex analytic sets):
\begin{theorem}[Gau-Lipman's Theorem]\label{complex_gau-lipman_thm}
Let $X ,Y\subset \C^n$ be two complex analytic sets. If there exists a homeomorphism $\varphi\colon (\C^n,X,0)\to (\C^n,Y,0)$ such that $\varphi$ and $\varphi^{-1}$ have a derivative at the origin (as mappings from $(\R^{2n},0)$ to $(\R^{2n},0)$), then $m(X,0)= m(Y,0)$.
\end{theorem}
This result was a generalization of the result proved separately by R. Ephraim in \cite{Ephraim:1976} and D. Trotman in \cite{Trotman:1977} (see also \cite{Trotman:1998}). They showed that the following question has a positive answer when the homeomorphism $\varphi$ is a $C^1$ diffeomorphism.

\begin{enumerate}[leftmargin=*]\label{zariski}
\item[]{\bf Question A.} Let $f,g\colon (\C^n,0)\to (\C,0)$ be two complex analytic functions. If there is a homeomorphism $\varphi\colon (\C^n,V(f),0)\to (\C^n,V(g),0)$, then is it true that $m(V(f),0)=m(V(g),0)$?
\end{enumerate} 
This question was asked by O. Zariski in 1971 (see \cite{Zariski:1971}) and  in its stated version is known as Zariski's multiplicity conjecture. It is still an open problem. 

Here, we are interested in the case of real analytic sets, however, the problem has a negative answer in this case, as we can see in the following example.
\begin{example}\label{ex_zariski}
Let $X=\{(x,y)\in\R^2;\, y=0\}$, $Y=\{(x,y)\in\R^2;\, y^3=x^2\}$ and $\varphi: \R^2\to \R^2$ given by $\varphi(x,y)=(x,x^{\frac{2}{3}}-y)$. Then, $\varphi$ is a homeomorphism such that $\varphi(X)=Y$, but $m(X)\equiv 1\, {\rm mod\,}2$ and $m(Y)\equiv 0\, {\rm mod\,} 2$.
\end{example}

However, some authors approached Question A in the real case. For example, J.-J. Risler in \cite{Risler:2001} proved that multiplicity ${\rm mod\,} 2$ of a real analytic curve is invariant by bi-Lipschitz homeomorphisms. T. Fukui, K. Kurdyka and L. Paunescu in \cite{FukuiKP:2004} proposed the following conjecture
\begin{enumerate}[leftmargin=*]
\item[]{\bf Conjecture F-K-P.} Let $h \colon (\R^n, 0) \to (\R^n,0)$ be the germ of a subanalytic, arc-analytic, bi-Lipschitz homeomorphism, and let $X, Y\subset\R^n$ be two irreducible analytic germs. Suppose that $Y = h(X)$, then $m(X) = m(Y)$.
\end{enumerate} 
and proved that multiplicity of a real analytic curve is invariant by arc-analytic bi-Lipschitz homeomorphisms. G. Valette in \cite{Valette:2010} proved that the multiplicity ${\rm mod\,} 2$ of a real analytic hypersurface is invariant by arc-analytic bi-Lipschitz homeomorphisms and the multiplicity ${\rm mod\,} 2$ of a real analytic surface is invariant by subanalytic bi-Lipschitz homeomorphisms and the author in \cite{Sampaio:2020} proved that the multiplicity ${\rm mod\,} 2$ of a real analytic surface is invariant by bi-Lipschitz homeomorphisms.

The main aim of this paper is to prove the real version of Gau-Lipman's Theorem, i.e., it is to prove that the multiplicity ${\rm mod\,} 2$ of real analytic sets is a differential invariant (see Corollary \ref{gau-lipman_thm}). Let us remark that Y.-N. Gau and J. Lipman's proof does not work in the real setting, since their proof uses, for instance, that the tangent cone at a point of a complex analytic set is a complex algebraic set, which may not happen for tangent cones of real analytic sets.

Let us describe how this paper is organized. In Section \ref{preliminaries}, we present some preliminaries. In Section \ref{main_sec}, we present a result on differential invariance of the multiplicity of real analytic sets (see Theorem \ref{main_result}) and as a corollary, we obtain the real version of Gau-Lipman's Theorem (see Corollary \ref{gau-lipman_thm}) and we present also some examples in order to show that the hypotheses of Theorem \ref{main_result} cannot be removed. In Section \ref{main_sec_two}, we present a generalization of Gau-Lipman's Theorem (see Theorem \ref{generalization_gau-lipman_thm}), which is the complex version of Theorem \ref{main_result}. An example showing that the hypotheses in Theorem \ref{generalization_gau-lipman_thm} are weaker than the hypotheses in Gau-Lipman's Theorem is also presented (see Example \ref{weaker_hyp_gau_lipman}).

\bigskip

% \noindent{\bf Acknowledgments}. The author wishes to thank my thesis advisors Alexandre Fernandes and Lev Birbrair by incentive and interest in this research. Some results of the paper are part of the author's PhD thesis at Universidade Federal do Cear\'a. 

\section{Preliminaries}\label{preliminaries}
Here, all real analytic sets are supposed to be pure dimensional.
\begin{definition}
Let $X\subset \R^n$ be a subset such that $x_0\in \overline{X}$.
We say that $v\in \R^n$ is a tangent vector of $X$ at $x_0\in\R^n$ if there is a sequence of points $\{x_i\}\subset X$ tending to $x_0\in \R^n$ and there is a sequence of positive numbers $\{t_i\}\subset\R^+$ such that 
$$\lim\limits_{i\to \infty} \frac{1}{t_i}(x_i-x_0)= v.$$
Let $C(X,x_0)$ denote the set of all tangent vectors of $X$ at $x_0\in \R^n$. We call $C(X,x_0)$ the {\bf tangent cone} of $X$ at $x_0$.
\end{definition}
\begin{remark}{\rm It follows from the curve selection lemma for subanalytic sets that, if $X\subset \R^n$ is a subanalytic set and $x_0\in \overline{X}$ is a non-isolated point, then the following holds true }
\begin{eqnarray*}
C(X,x_0)=\{v;\, \exists\, \mbox{ subanalytic }\alpha:[0,\varepsilon )\to \R^n\,\, \mbox{s.t.}\,\, \alpha(0)=x_0,\, \alpha((0,\varepsilon ))\subset X\,\, \mbox{and}\,\,\\ 
\alpha(t)-x_0=tv+o(t)\}. 
\end{eqnarray*}
\end{remark}
\begin{definition}
The mapping $\beta _n:\mathbb{S}^{n-1}\times \R^+\to \R^n$ given by $\beta_n(x,r)=rx$ is called {\bf spherical blowing-up} (at the origin) of $\R^n$.
\end{definition}
Note that $\beta _n:\mathbb{S}^{n-1}\times (0,+\infty )\to \R^n\setminus \{0\}$ is a homeomorphism with inverse $\beta_n^{-1}:\R^n\setminus \{0\}\to \mathbb{S}^{n-1}\times (0,+\infty )$ given by $\beta_n^{-1}(x)=(\frac{x}{\|x\|},\|x\|)$.
\begin{definition}
The {\bf strict transform} of the subset $X$ under the spherical blowing-up $\beta_n$ is $X':=\overline{\beta_n^{-1}(X\setminus \{0\})}$ and the {\bf boundary} $\partial X'$ of the {\bf strict transform} is $\partial X':=X'\cap (\mathbb{S}^{n-1}\times \{0\})$.
\end{definition}
Remark that $\partial X'=C_X\times \{0\}$, where $C_X=C(X,0)\cap \mathbb{S}^{n-1}$.

\subsection{Multiplicity and relative multiplicities}\label{section:mainresult}

% We denote by $f_{\C}$ the complexification of a real analytic function $f$.

Let $X\subset \R^{n}$ be a $d$-dimensional real analytic set with $0\in X$ and 
$$
X_{\C}= V(\mathcal{I}_{\R}(X,0)),
$$
where $\mathcal{I}_{\R}(X,0)$ is the ideal in $\mathbb{C}\{z_1,...,z_n\}$ generated by the complexifications of all germs of real analytic functions that vanish on the germ $(X,0)$. We have that $X_{\C}$ is a germ of a complex analytic set and $\dim_{\C}X_{\C}=\dim_{\R}X$ (see \cite[Propositions 1 and 3, pp. 91-93]{Narasimhan:1966}). Then, for a linear projection $\pi:\C^{n}\to\C^d$ such that $\pi^{-1}(0)\cap C(X_{\C},0) =\{0\}$, there exists an open neighborhood $U\subset \C^n$ of $0$ such that $\# (\pi^{-1}(x)\cap (X_{\C}\cap U))$ is constant for a generic point $x\in \pi(U)\subset\C^d$. This number is the multiplicity of $X_{\C}$ at the origin and it is denoted by $m(X_{\C},0)$.

\begin{definition}
With the above notation, we define the multiplicity of $X$ at the origin by $m(X):=m(X_{\C},0)$.
\end{definition}

\begin{definition}
We shall not distinguish between a $2(n-d)$-dimensional real linear subspace in $\C^n$ and its canonical image in $G^{2n}_{2(n-d)}(\R)$. Thus, we regard $G^{n}_{n-d}(\C)$ as a subset of $G^{2n}_{2(n-d)}(\R)$.
Let $\mathcal{E}(X_{\C})$ denote the subset of $G^{2n}_{2(n-d)}(\R)$ consisting of all $L\in G^{2n}_{2(n-d)}(\R)$ such that $L\cap C(X_{\C},0)=\{0\}$.
\end{definition}

\begin{remark}\label{transversal-cones}
We have the following comments on the set  $\mathcal{E}(X_{\C})$.
\begin{enumerate}
 \item [(i)]\label{transversal-cones-i} $\mathcal{E}(X_{\C})$ is an open dense set in $G^{2n}_{2(n-d)}(\R)\cong G^{2n}_{2d}(\R)$ (see \cite[Lemme 1.4]{Comte:2000});
 \item [(ii)]\label{transversal-cones-ii} For each $L \in \mathcal{E}(X_{\C})\cap G^{n}_{n-d}(\C)$, let $\pi_L\colon \C^n\to L^{\perp}$ be the orthogonal projection over $L$. Then, there exist a polydisc $U\subset \C^n$ and a complex analytic set $\sigma\subset U':=\pi_L(U)$ such that $\dim \sigma <\dim X_{\C}$ and $\pi_L\colon (U\cap X_{\C})\setminus \pi_L^{-1}(\sigma)\to U'\setminus \sigma$ is a $k$-sheeted cover with $k=m(X_{\C},0)$ (see \cite[Theorem 7P, p. 234]{Whitney:1972});
 \item [(iii)]\label{transversal-cones-iii} Since $\pi:=\pi_L$ is an $\R$-linear mapping, we identify the $d$-dimensional real linear subspace $\pi(\R^n)$ with $\R^d$ and, with this identification, we obtain that $\R^d\cap \sigma$ is a closed nowhere dense subset of $\R^d\cap U'$. Indeed, it is clear that $\R^d\cap \sigma$ is a closed subset of $\R^d\cap U'$ and, thus, if $\sigma$ is somewhere dense in $\R^d\cap U'$, then $\sigma$ contains an open ball  $B_r(p)\subset\R^d\cap U'$, which implies that $\sigma$ must contain a non-empty open subset of $U'$ (see \cite[Proposition 1, p. 91]{Narasimhan:1966}) and, thus, we obtain a contradiction. Therefore, $\sigma$ is nowhere dense in $\R^d\cap U'$ and, then, $\R^d\cap U'\setminus \sigma$ is an open dense subset of $\R^d\cap U'$;
 \item [(iv)]\label{transversal-cones-iv} For a generic point $x\in \R^d$ near to the origin (i.e., for $x\in (\R^d\cap U')\setminus \sigma$), we have
\begin{eqnarray*}
m(X_{\C},0)&=&\# (\pi^{-1}(x)\cap (X_{\C}\cap U))\\
           &=&\# (\R^n\cap \pi^{-1}(x)\cap (X_{\C}\cap U))+\# ((\C^n\setminus \R^n)\cap \pi^{-1}(x)\cap (X_{\C}\cap U))\\
           &=&\# (\pi^{-1}(x)\cap (X\cap U))+\# (\pi^{-1}(x)\cap ((X_{\C}\setminus \R^n)\cap U)).
\end{eqnarray*}
Since for each $f\in \mathcal{I}_{\R}(X,0)$, we may write $f(z)=\sum \limits _{|I|=k}^{\infty}a_Iz^I$ such that $a_I\in\R$ for all $I$, then $f(z_1,...,z_n)=0$ if and only if $f(\bar z_1,...,\bar z_n)=0$, where each $\bar z_i$ denotes the complex conjugate of $z_i$. In particular, $\# (\pi^{-1}(x)\cap ((X_{\C}\setminus \R^n)\cap U))$ is an even number. Therefore, we obtain that
$m(X)\equiv \# (\pi^{-1}(x)\cap (X\cap U)) \,{\rm mod\,} 2$ for a generic point $x\in \R^d$ near to the origin. 
\end{enumerate}
\end{remark}

\begin{definition}
Let $X\subset \R^n$ be a subanalytic set such that $0\in \overline X$ is a non-isolated point. We say that $x\in\partial X'$ is {\bf a simple point of $\partial X'$}, if there is an open set $U\subset \R^{n+1}$ with $x\in U$ such that:
\begin{itemize}
\item [a)] the connected components of $(X'\cap U)\setminus \partial X'$, say $M_1,..., M_r$, are topological manifolds with $\dim M_i=\dim X$, $i=1,...,r$;
\item [b)] $(M_i\cup \partial X')\cap U$ are topological manifolds with boundary. 
\end{itemize}
Let $Smp(\partial X')$ be the set of simple points of $\partial X'$.
\end{definition}

\begin{remark}\label{density_top}
By Theorems 2.1 and 2.2 in \cite{Pawlucki:1985}, we obtain that ${\rm Smp}(\partial X')$ is an open dense subset of the $(d-1)$-dimensional part of $\partial X'$ whenever $\partial X'$ is a $(d-1)$-dimensional subset, where $d=\dim X$.
\end{remark}
\begin{definition}
Let $X\subset \R^n$ be a subanalytic set such that $0\in X$.
We define $k_X:Smp(\partial X')\to \N$ such that $k_X(x)$ is the number of connected components of the germ $(\beta_n^{-1}(X\setminus\{0\}),x)$.
\end{definition}
\begin{remark}\label{locally-constant}
It is clear that the function $k_X$ is locally constant. In fact, $k_X$ is constant in each connected component $C_j$ of $Smp(\partial X')$. Then, we define $k_X(C_j):=k_X(x)$ with $x\in C_j$.
\end{remark}
\begin{remark}
The numbers $k_X(C_j)$ are equal to the numbers $n_j$ defined by Kurdyka and Raby \cite{Kurdyka:1989}, p. 762.
\end{remark}
\begin{remark}
When $X$ is a complex analytic set, there is a complex analytic set $\Sigma$ with $\dim \Sigma <\dim X$, such that $X_j\setminus \Sigma$ intersects only one connected component $C_i$ of $Smp(\partial X')$ (see \cite{Chirka:1989}, pp. 132-133), for each irreducible component $X_j$ of the tangent cone $C(X,0)$. Then we define $k_X(X_j):=k_X(C_i)$.
\end{remark}
\begin{remark}[{\cite[p. 133, Proposition]{Chirka:1989}}]\label{multip}
{\rm Let $X$ be a complex analytic set of $\C^n$ with $0\in X$ and let $X_1,...,X_r$ be the irreducible components of $C(X,0)$. Then} 
\begin{equation*}\label{mult_rel}
m(X,0)=\sum_{j=1}^rk_X(X_j)\cdot m(X_j,0).
\end{equation*}
\end{remark}
\begin{definition}
Let $X\subset \R^{n}$ be a real analytic set with $0\in X$. We denote by $C_X'$ the union of all connected components $C_j$ of $Smp(\partial X')$ having odd $k_X(C_j)$. We call $C_X'$ the {\bf odd part of $C_X\subset \mathbb{S}^{n-1}$}.
\end{definition}

\begin{definition}\label{multiplicity-odd-cone}
Let $X\subset \R^{n}$ be a $d$-dimensional real analytic set with $0\in X$, $L \in \mathcal{E}(X_{\C})\cap G^{n}_{n-d}(\C)$, let $\pi:=\pi_L\colon \C^n\to L^{\perp}$ be the orthogonal projection over $L$. Let $\pi':\mathbb{S}^{n-1}\setminus L\to \mathbb{S}^{d-1}$ be the mapping given by $\pi'(u)=\frac{\pi(u)}{\|\pi(u)\|}$, where we are identifying $\pi(\R^n)$ with $\R^d$ and $\pi(\R^n)\cap \mathbb{S}^{2n-1}$ with $\mathbb{S}^{d-1}$ (see Remark \ref{transversal-cones} (iii)). We define
$$\varphi_{\pi,C_X'}(x):=\#(\pi'^{-1}(x)\cap C_X').$$
In this case, if $\varphi_{\pi,C_X'}(x) \,{\rm mod\,} 2$ is constant for a generic $x\in\mathbb{S}^{d-1}$, we write $m_{\pi}(C_X'):=\varphi_{\pi,C_X'}(x)\, {\rm mod\,} 2$, for a generic $x\in\mathbb{S}^{d-1}$.
\end{definition}

% 
% \begin{definition}
% Let $X\subset \mathbb{R}^n$ and $Y\subset \mathbb{R}^m$ be closed subsets such that $0\in X\times Y$. We say that $(X,0)$ and $(Y,0)$ are {\bf differentiable equivalent} if there exists a homeomorphism $\varphi\colon (X,0)\to (Y,0)$ such that $\varphi$ and $\varphi^{-1}$ are differentiable at $0$. In this case, we say that $\varphi$ is a {\bf differentiable equivalence} (between $(X,0)$ and $(Y,0)$).
% \end{definition}

\section{Proof of the real version of Gau-Lipman's Theorem}\label{main_sec}
In this Section, we show that the multiplicity ${\rm mod\,} 2$ of a real analytic set is a differential invariant, which is the real version of Gau-Lipman's Theorem. In fact, we prove a little bit more, as we can see in the next result.

\begin{theorem}\label{main_result}
Let $X, Y \subset  \R^N$ be two real analytic sets  with $0\in X\cap Y$. Assume that there exists a mapping $\varphi\colon (\R^N,0)\to (\R^N,0)$ such that $\varphi\colon (X,0)\to (Y,0)$ is a homeomorphism. If $\varphi$ has a derivative at the origin and $D\varphi_0\colon \R^N\to \R^N$ is an isomorphism, then $m(X)\equiv m(Y)\, {\rm mod\,} 2$.
\end{theorem}
\begin{proof}
Since $\phi:=D\varphi_0:\R^{N} \to \R^{N}$ is an $\R$-linear isomorphism, we have that $A=\phi(X)$ is a real analytic set.

We have that the complexification of $\phi$, denoted by $\phi_{\C}$, is a complex diffeomorphism between $X_{\C}$ and $A_{\C}$. Thus, by Proposition in (\cite{Chirka:1989}, Section 11, p. 120), $m(X_{\C},0)=m(A_{\C},0)$. Therefore, $m(X)=m(A)$.

Thus, it is enough to show that $m(Y)\equiv m(A)\, {\rm mod\,} 2$. In order to do this, we consider the mapping $\psi\colon (Y,0)\to (A,0)$ given by $\psi=\phi\circ \varphi^{-1}$. 

\begin{claim}\label{homeo_bs}
The mapping $\psi': Y'\to A'$ given by
$$
\psi'(x,t)=\left\{\begin{array}{ll}
\left(\frac{\psi(tx)}{\|\psi(tx)\|},\|\psi(tx)\|\right),& t\not=0\\
(x,0),& t=0,
\end{array}\right.
$$
is a homeomorphism.
\end{claim}
\begin{proof}[Proof of Claim \ref{homeo_bs}]
Observe that $\nu\colon\mathbb{S}^{N-1}\to \mathbb{S}^{N-1}$ given by $$\nu(x)=\frac{\phi(x)}{\|\phi(x)\|}$$ 
is a homeomorphism and using that $\varphi(tx)=t\phi(x)+o(t)$, we obtain
$$
\lim\limits _{t\to 0^+}\frac{\varphi(tx)}{\|\varphi(tx)\|}=\frac{\phi(x)}{\|\phi(x)\|}=\nu(x).
$$
Therefore, the mappings $\phi': \mathbb{S}^{N-1}\times [0,\infty)\to \mathbb{S}^{N-1}\times [0,\infty)$ and $\varphi'\colon X'\to Y '$ given by
$$
\phi'(x,t)=\left\{\begin{array}{ll}
\left(\frac{\phi(tx)}{\|\phi(tx)\|},\|\phi(tx)\|\right),& t\not=0\\
(\nu(x),0),& t=0
\end{array}\right.
$$
and
$$
\varphi'(x,t)=\left\{\begin{array}{ll}
\left(\frac{\varphi(tx)}{\|\varphi(tx)\|},\|\varphi(tx)\|\right),& t\not=0\\
(\nu(x),0),& t=0
\end{array}\right.
$$
are homeomorphisms, which implies that the mapping $(\varphi^{-1})': Y'\to X '$ given by
$$
(\varphi^{-1})'(x,t)=\left\{\begin{array}{ll}
\left(\frac{\varphi^{-1}(tx)}{\|\varphi^{-1}(tx)\|},\|\varphi^{-1}(tx)\|\right),& t\not=0\\
(\nu^{-1}(x),0),& t=0,
\end{array}\right.
$$
is also a homeomorphism. 
Since $\psi'=\phi'\circ (\varphi^{-1})'$, we finish the proof of Claim \ref{homeo_bs}.
\end{proof}
As a direct consequence, we obtain that $Smp(\partial Y')=\psi'(Smp(\partial Y'))=Smp(\partial A')$. 
\begin{claim}\label{preserves_odd_cone}
 $k_Y(p)=k_{A}(p)$ for all $p\in Smp(\partial Y')$. 
\end{claim}
\begin{proof}[Proof of Claim \ref{preserves_odd_cone}]
In fact, let $p\in Smp(\partial Y')$ be a point and let $U \subset Y'$ be a small neighborhood of $p$. Since $\psi':Y'\to A'$ is a homeomorphism, we have that $V=\psi'(U)$ is a small neighborhood of $p=\psi'(p)\in \partial A'$. Moreover, $\psi'(U\setminus \partial Y')=V\setminus \partial A'$, since $\psi'|_{\partial Y'}:\partial Y'\to \partial A'$ is a homeomorphism, as well. Using once more that $\psi'$ is a homeomorphism, we obtain that the number of connected components of $U\setminus \partial Y'$ is equal to the number of connected components of $V\setminus \partial A'$, showing that $k_Y(p)=k_{A}(p)$ for all $p\in Smp(\partial Y')$.
\end{proof}

As a direct consequence, we obtain that $C_Y'=\psi'(C_Y')=C_{A}'$.

Let $L \in \mathcal{E}(Y_{\C})\cap G^{N}_{N-d}(\C)$ and let $\pi:=\pi_L\colon \C^N\to L^{\perp}$ be the orthogonal projection over $L$, where $d=\dim Y$ (see Remark \ref{transversal-cones}). Let $\pi':\mathbb{S}^{N-1}\setminus L\to \mathbb{S}^{d-1}$ be given by $\pi'(u)=\frac{\pi(u)}{\|\pi(u)\|}$, where we are identifying $\pi(\R^N)$ with $\R^d$ and $\pi(\R^N)\cap \mathbb{S}^{2N-1}$ with $\mathbb{S}^{d-1}$ as in Definition \ref{multiplicity-odd-cone}.

\vspace{0.3cm}

\begin{claim}\label{mult_odd_cone}
$\varphi_{\pi,C_Y'}(y)=\#(\pi'^{-1}(y)\cap C_Y')\, {\rm mod\,} 2$ is constant for a generic point $y\in\mathbb{S}^{d-1}$. Moreover, $m_{\pi}(C_Y')\equiv m(Y)\, {\rm mod\,} 2$.
\end{claim}
\begin{proof}[Proof of Claim \ref{mult_odd_cone}]
If $\dim C_Y<d-1$ then $C_Y'=\emptyset$ and $\dim C(\pi(Y),0)<d$, which implies that there exist $w\in \mathbb{S}^{d-1}$ and small enough numbers $\eta,\varepsilon\in (0,1)$ such that $C_{\eta,\varepsilon }(y)\cap \pi(Y)=\emptyset $, where $C_{\eta,\varepsilon }(w)=\{v\in \R^d;\, \|v-tw\|\leq \eta t, \,t\in(0,\varepsilon]\}$. Therefore $\varphi_{\pi,C_Y'}(y)=0$ for any point $y\in\mathbb{S}^{d-1}$ and $m(Y)\equiv 0\, {\rm mod\,} 2$, since $C_Y'=\emptyset$ and $\pi^{-1}(v)\cap Y=\emptyset$, for all $v\in C_{\eta,\varepsilon }(w)$ (see Remark \ref{transversal-cones} (iv)). In particular, $m_{\pi}(C_Y')$ is defined and satisfies $m_{\pi}(C_Y')\equiv m(Y)\, {\rm mod\,} 2$.

Thus, we may assume that $\dim C_Y=d-1$. By Remark \ref{density_top}, $Smp(\partial Y')$ is an open dense subset of the $(d-1)$-dimensional part of $\partial Y'=C_Y\times \{0\}\cong C_Y$.
Let $y\in \mathbb{S}^{d-1}$ be a generic point such that $\pi'^{-1}(y)\cap C_Y=\pi'^{-1}(y)\cap Smp(\partial Y')=\{y_1,...,y_p\}$ and $u=\#(\pi^{-1}(ty)\cap Y)\equiv m(Y) {\rm mod}\,2$, for all small enough $t>0$ (see Remark \ref{transversal-cones} (iv)). Then, we have the following
\begin{equation*}
u= \sum\limits _{j=1}^p k_Y(y_j).
\end{equation*}

In fact, let $\eta,\varepsilon>0$ be small enough numbers such that $C_{\eta,\varepsilon }(y)\cap \pi(br(\pi|_Y))=\emptyset $, where $C_{\eta,\varepsilon }(y)=\{v\in \R^d;\, \|v-ty\|\leq \eta t, \,t\in(0,\varepsilon]\}$ and $br(\pi|_Y)$ denotes the set of all critical points of $\pi|_Y$. Thus, denote the connected components of $(\pi|_Y)^{-1}(C_{\eta,\varepsilon }(y))$ by $Y_1,...,Y_u$. Hence, $\pi|_{Y_i}:Y_i\to C_{\eta,\varepsilon }(y)$ is a homeomorphism, for $i=1,...,u$. Thus, for each $i=1,...,u$, there is a unique $\gamma_i\colon (0,\varepsilon)\to Y_i$ such that $\pi(\gamma_i(t))=ty$ for all $t\in (0,\varepsilon)$. We define for each $i=1,...,u$, $\widetilde\gamma_i\colon [0,\varepsilon)\to \overline{\beta_{N}^{-1}(Y_i)}$ given by $\widetilde \gamma_i(s)=\lim\limits_{t\to s^+}\beta_{N}^{-1}\circ \gamma_i(t)$, for all $s\in [0,\varepsilon)$.

We remark that $\widetilde \gamma_i(0)=\lim\limits _{t\to 0^+}\widetilde \gamma_i(t)\in \{y_1,...,y_p\}$, for all $i=1,...,u$ and, thus, $u\leq \sum\limits _{j=1}^p k_Y(y_j)$. Shrinking $\eta$, if necessary, we can suppose that each $C_{Y_i}$ contains at most one $y_j$. Thus for fixed $y_j$ and if $\gamma:[0,\delta)\to Y$ is a subanalytic curve such that $\lim\limits _{t\to 0^+}\beta_{N}^{-1}\circ \gamma(t)=y_j$, then there exists $\delta_0>0$ such that $\pi(\gamma(t))\in C_{\eta,\varepsilon }(y)$, for all $0<t<\delta_0$. So, there is $i\in \{1,...,u\}$ such that $\gamma(t)\in Y_i$, with $0<t<\delta_0$. Then, $\widetilde \gamma_i(0)=y_j$ and we obtain the equality $u=\sum\limits _{j=1}^p k_Y(y_j)$. 

Let $C_1,...,C_r$ be the connected components of $Smp(\partial Y')$. By Remark \ref{locally-constant}, we know that $k_Y$ is constant in each $C_i$ and, thus, if $y_j, y_{j'}\in C_i$ then $k_Y(y_j)=k_Y(y_{j'})$. Since $\pi'^{-1}(y)\cap C_Y=\pi'^{-1}(y)\cap Smp(\partial Y')=\{y_1,...,y_p\}$, we have
$$
u=\sum\limits _{j=1}^p k_Y(y_j)=\sum\limits _{i\in \Lambda} k_Y(C_i)\cdot \#(\pi'^{-1}(y)\cap C_i),
$$
where $\Lambda=\{i\in\{1,...,r\}; \pi'^{-1}(y)\cap C_i\not=\emptyset\}$.
Therefore, we obtain
$$
u=\sum\limits _{i=1}^r k_Y(C_i)\cdot \#(\pi'^{-1}(y)\cap C_i).
$$
However, $\sum\limits _{i=1}^r k_Y(C_i)\cdot \#(\pi'^{-1}(y)\cap C_i)\equiv \#(\pi'^{-1}(y)\cap C_Y')\, {\rm mod\,} 2$ and $u\equiv m(Y)\, {\rm mod\,} 2$, then
$$
m(Y)\equiv \#(\pi'^{-1}(y)\cap C_Y')\, {\rm mod\,} 2,
$$
for a generic $y\in\mathbb{S}^{d-1}$, which shows that $\varphi_{\pi,C_Y'}(y)=\#(\pi'^{-1}(y)\cap C_Y')\, {\rm mod\,} 2$ is constant for a generic point $y\in\mathbb{S}^{d-1}$ and, thus, $m_{\pi}(C_Y')$ is defined and satisfies $m_{\pi}(C_Y')\equiv m(Y)\, {\rm mod\,} 2$.
\end{proof}
Then, we obtain that $m_{\pi}(C_Y')$ does not depend on a generic $\pi$, since $m(Y)$ does not depend on a generic $\pi$. Similarly, we obtain that $m_{\bar{\pi}}(C_{A}')$ does not depend on a generic projection $\bar{\pi}$ and $m_{\bar{\pi}}(C_{A}')\equiv m({A})\, {\rm mod\,} 2$. Thus, we write $m(C_Y')$ (resp. $m(C_{A}')$) instead of $m_{\pi}(C_Y')$ (resp. $m_{\bar{\pi}}(C_{A}')$). 

Let $\tilde L \in \mathcal{E}(Y_{\C}\cup A_{\C})\cap G^{N}_{N-d}(\C)$ and let $\tilde\pi:=\pi_{\tilde L}\colon \C^N\to \tilde L^{\perp}$ be the orthogonal projection over $\tilde L$. Let $\tilde\pi':\mathbb{S}^{N-1}\setminus \tilde L\to \mathbb{S}^{d-1}$ given by $\tilde\pi'(u)=\frac{\tilde\pi(u)}{\|\tilde\pi(u)\|}$ as in Definition \ref{multiplicity-odd-cone}.  Then, for a generic $y\in \mathbb{S}^{d-1}$, we obtain the following
\begin{eqnarray*}
m(Y) &\equiv & m(C_Y')\, {\rm mod\,} 2\quad (\mbox{by Claim \ref{mult_odd_cone}})\\
     &\equiv & \#(\tilde\pi'^{-1}(y)\cap C_Y')\, {\rm mod\,} 2\quad (\mbox{by the definition of } m(C_Y'))\\
	 &\equiv & \#(\tilde\pi'^{-1}(y)\cap C_{A}')\, {\rm mod\,} 2\quad (\mbox{since } C_Y'=C_{A}')\\
	 &\equiv & m(C_{A}')\, {\rm mod\,} 2\quad (\mbox{by the definition of } m(C_{A}'))\\
	 &\equiv & m(A)\, {\rm mod\,} 2\quad (\mbox{by Claim \ref{mult_odd_cone}}),
\end{eqnarray*}
% Similarly, we obtain also that $m(C_{A}')$ does not depend on ${A}$.
% \end{proof}
% 
% Therefore, by Claims \ref{mult_odd_cone} and \ref{non-dependence}, we obtain
% $$
% m(Y)\equiv m(C_Y')\equiv m(C_{A}')\equiv m(A)\,{\rm mod\,} 2,
% $$
which finishes the proof.
\end{proof}
As consequences, we obtain the following.
\begin{corollary}\label{gau-lipman_thm}
Let $X ,Y\subset \R^N$ be two real analytic sets containing $0$. If there exists a homeomorphism $\varphi\colon (\R^N,X,0)\to (\R^N,Y,0)$ such that $\varphi$ and $\varphi^{-1}$ have a derivative at the origin, then $m(X)\equiv m(Y)\, {\rm mod\,} 2$.
\end{corollary}
\begin{proof}
Since $\varphi$ and $\varphi^{-1}$ have a derivative at $0$, we have that $D\varphi_0\colon \R^N\to \R^N$ is an isomorphism and by Theorem \ref{main_result}, $m(X)\equiv m(Y)\, {\rm mod\,} 2$.
\end{proof}

\begin{definition}
Let $X\subset \mathbb{R}^n$ and $Y\subset \mathbb{R}^m$ be closed subsets. We say that a continuous mapping $f:X\to Y$ is differentiable at $x\in X$, if there exist an open $U\subset \mathbb{R}^n$ and a continuous mapping $F:U\to \mathbb{R}^m$ such that $x\in U$, $F|_{X\cap U}=f|_{X\cap U}$ and $F$ has a derivative at $x$. 
\end{definition}

\begin{corollary}
Let $X \subset \R^m$ and  $Y\subset \R^n$ be two real analytic sets containing $0$. If there exists a homeomorphism $\phi\colon (X,0)\to (Y,0)$ such that $\phi$ and $\phi^{-1}$ are differentiable at $0$, then $m(X)\equiv m(Y)\, {\rm mod\,} 2$.
\end{corollary}
\begin{proof}
By hypothesis there are closed representatives $A$ and $B$ respectively of $(X,0)$ and $(Y,0)$ and a homeomorphism $\phi\colon A\to B$ such that $\phi(0)=0$ and, $\phi$ and $\phi^{-1}$ have a derivative at $0$. Let $\widetilde \phi\colon \R^{m} \to \R^{n}$ (resp. $\widetilde \psi\colon \R^{n} \to \R^{m}$) be a continuous extension of $\phi$ (resp. $\phi^{-1}$), which has a derivative at $0\in \R^m$ (resp. $0\in \R^n$). Then the mapping $\varphi\colon \R^{m+n}\to \R^{m+n}$ given by
$$
\varphi(x,y)=(x-\widetilde\psi(y+\widetilde\phi(x)),y+\widetilde\phi(x))
$$
is a homeomorphism such that $\varphi(A\times \{0\})=\{0\}\times B$, its inverse is given by
$$
\varphi^{-1}(z,w)=(z+\widetilde\psi(w), w-\widetilde\phi(z+\widetilde\psi(w))).
$$
and both have a derivative at $0\in \R^{m+n}$.

Since $m(A\times \{0\}) = m(A)= m(X)$ and $m(\{0\}\times B)= m(B)= m(Y)$, by Corollary \ref{gau-lipman_thm}, we obtain $m(X)\equiv m(Y)\, {\rm mod\,} 2$.
\end{proof}
Let us make some remarks on Theorem \ref{main_result}. Firstly, the assumption that $D\varphi_0$ is an isomorphism cannot be removed, as it is shown in the next example.
\begin{example}
Let $X=\{(x,y)\in\R^2 ;\, y^3=x^2\}$ and $Y=\{(x,y)\in\R^2;\, y=0\}$. Then $\varphi\colon (\R^2,X,0)\to (\R^2,Y,0)$ given by $\varphi(x,y)=(x,y^3-x^2)$ is a homeomorphism, which has a derivative at the origin, %but its inverse $\varphi^{-1}(u,v)=(u,(v+u^2)^{\frac{1}{3}})$ is not differentiable at the origin. 
but $D\varphi_0$ is not an isomorphism.
In this case, $m(X)=2$ and $m(Y)=1$.
\end{example}
Secondly, we cannot expect equality (without modulus 2) as is shown in the next example.
\begin{example}
Let $V=\{(x,y,z)\in\R^3;\, z^3=x^5y+xy^5\}$. Then the mapping $\varphi\colon \R^3\to \R^3$ given by $\varphi(x,y,z)=(x,y,z-(x^5y+xy^5)^{\frac{1}{3}})$ is a homeomorphism which has a derivative at the origin and its inverse has also a derivative at the origin. Moreover, $\varphi(V)=\R^2\times \{0\}$, but $m(V)=3$ and $m(\R^2\times \{0\})=1$.
\end{example}

We finish this Section by presenting an example of a mapping which has a derivative at the origin and is a homeomorphism between two analytic sets, but its inverse has not a derivative at the origin.

\begin{example}
The mapping $\varphi\colon \mathbb{R}^2\to \mathbb{R}^2$  given by
$$
\varphi(x,y)=\left\{\begin{array}{ll}
\big(x,y+2y^2\sin{\frac{1}{y}}\big),& y\not=0\\
(x,0),& y=0
\end{array}\right.
$$
has a derivative at the origin, $D\varphi_0=id\colon \R^2\to \R^2$ and $\varphi|_{\R\times \{0\}}\colon \R\times \{0\}\to \R\times \{0\}$ is a homeomorphism, but it does not have an inverse which has a derivative at the origin.
\end{example}

\section{A generalization of Gau-Lipman's Theorem}\label{main_sec_two}
In this Section, we present a complex version of Theorem \ref{main_result}, which is a generalization of Gau-Lipman's Theorem.

\begin{theorem}\label{generalization_gau-lipman_thm}
Let $X, Y \subset  \C^N$ be two complex analytic sets with $0\in X\cap Y$. Assume that there exists a mapping $\varphi\colon (\C^N,0)\to (\C^N,0)$ such that $\varphi|_X\colon (X,0)\to (Y,0)$ is a homeomorphism. If $\varphi$ has a derivative at the origin (as a mapping from $(\R^{2N},0)$ to $(\R^{2N},0)$) and $D\varphi_0\colon \R^{2N}\to \R^{2N}$ is an isomorphism, then $m(X,0)= m(Y,0)$.
\end{theorem}
\begin{proof}
By using that $\phi:=D\varphi_0\colon\R^{2N}\to\R^{2N}$ is an $\R$-linear isomorphism, we obtain that $\phi$ maps bijectively the irreducible components of $C(X,0)$ over the irreducible components of $C(Y,0)$ (see Lemma A.8 in \cite{Gau-Lipman:1983} or Proposition 2 in \cite{Sampaio:2019}) and the mapping $\varphi': X'\to Y'$ given by
$$
\varphi'(x,t)=\left\{\begin{array}{ll}
\left(\frac{\varphi(tx)}{\|\varphi(tx)\|},\|\varphi(tx)\|\right),& t\not=0\\
\left(\frac{\phi(x)}{\|\phi(x)\|},0\right),& t=0,
\end{array}\right.
$$
is a homeomorphism. 
Let $X_1,...,X_r$ and $Y_1,...,Y_r$ be the irreducible components of $C(X,0)$ and $C(Y,0)$, respectively, such that $Y_j=\phi(X_j)$, $j=1,...,r$. Thus, by proceeding like in the proof of Claim \ref{preserves_odd_cone}, we obtain $k_X(X_j)=k_Y(Y_j)$, for all $j=1,...,r$.

Fixing $j\in \{1,...,r\}$, and by looking at $X_j$ and $Y_j$ as real algebraic sets in $\R^{2N}\cong\C^N$ and, as $\phi$ is an $\R$-linear isomorphism, then its complexification, $\phi_{\C}\colon\C^{2N}\to\C^{2N}$, is a $\C$-linear isomorphism such that $\phi_{\C}(X_{j\C})=Y_{j\C}$. By Proposition 2.9 in \cite{Ephraim:1976b}, $X_{j\C}$ (resp. $Y_{j\C}$)  is complex analytic diffeomorphic to $X_j\times c_N(X_j)$ (resp. $Y_j\times c_N(Y_j)$), where
$c_N:\C^N\to \C^N$ is the conjugation mapping given by $c_N(z_1,...,z_N)=(\overline{z}_1,...,\overline{z}_N)$. 
Then, 
$$m(X_{j\C},0)= m(X_j\times c_N(X_j),0)=m(Y_j\times c_N(Y_j),0)=m(Y_{j\C},0),$$ 
since the multiplicity is invariant by complex analytic diffeomorphisms (see \cite[Section 11, p. 120, Proposition]{Chirka:1989}). However, $c_N(X_j)$  and $c_N(Y_j)$ are complex analytic sets satisfying $m(c_N(X_j),0)=m(X_j,0)$ and $m(c_N(Y_j),0)=m(Y_j,0)$, then we obtain $m(X_j\times c_N(X_j),0)=m(X_j,0)^2$ and $m(Y_j\times c_N(Y_j),0)=m(Y_j,0)^2$, so we obtain $m(X_j,0)=m(Y_j,0)$, for all $j\in\{1,...,r\}$. 

By Remark \ref{multip}, 
$$m(X,0)=\sum_{j=1}^r k_X(X_j)\cdot m(X_j,0)$$
and 
$$m(Y,0)=\sum_{j=1}^r k_Y(Y_j)\cdot m(Y_j,0).$$

Therefore, $m(X,0)=m(Y,0)$.
\end{proof}

It is clear that as a consequence of Theorem \ref{generalization_gau-lipman_thm}, we obtain Gau-Lipman's Theorem. The next example shows that Theorem \ref{generalization_gau-lipman_thm} is really a generalization of Gau-Lipman's Theorem.

\begin{example}\label{weaker_hyp_gau_lipman}
Let $X=\{(x,y)\in\C^2;\, y^4-2x^3y^2-4x^5y+x^6-x^7=0\}$ and $\widetilde{X}=\{(x,y)\in\C^2;\, y^4-2x^3y^2-4x^6y+x^6-x^9=0\}$. The mapping $\Phi\colon (\C,0)\to (X,0)$ given by $\Phi(t)= (t^4,t^6+t^7)$ is a Puiseux parametrization of $X$ and there exists a complex analytic function $\phi\colon (\C,0)\to (\C,0)$ such that $ord_0(\phi)>9$ and the mapping $\tilde\Phi\colon (\C,0)\to (\widetilde{X},0)$ given by $\widetilde{\Phi}(t)= (t^4,t^6+t^9+\phi(t))$ is a Puiseux parametrization of $\widetilde{X}$.
Let $f\colon \mathbb{R}\to \mathbb{R}$ be the function given by
$$
f(s)=\left\{\begin{array}{ll}
s+2s^2\sin{\frac{1}{s}},& s\not=0\\
0,& s=0
\end{array}\right.
$$
and $\varphi\colon (\mathbb{C}^2,0)\to (\mathbb{C}^2,0)$ be the mapping given by
$$
\varphi(x,y)=\left\{\begin{array}{ll}
\widetilde{\Phi}(t),& \mbox{ if } (x,y)=\Phi(t)\mbox{ for some }t\in \C\\
\big(x,f\big(\frac{y+\overline{y}}{2}\big)+if\big(\frac{y-\overline{y}}{2}\big)\big),&  \mbox{ if } (x,y)\not=\Phi(t)\mbox{ for any }t\in \C.
\end{array}\right.
$$
Thus, $\varphi$
has a derivative at the origin, $D\varphi_0=id\colon \R^4\to \R^4$ and $\varphi|_{X}\colon (X,0)\to (\widetilde{X},0)$ is a homeomorphism. Moreover, since $X$ and $\widetilde{X}$ have different Puiseux pairs, there is no homeomorphism $h\colon (\mathbb{C}^2,0)\to (\mathbb{C}^2,0)$ such that $h(X)=\widetilde{X}$.
\end{example}

\noindent {\bf Acknowledgements.}
The author would like to thank the anonymous referees for their useful comments.

\end{document}